\documentclass{amsart}

\usepackage[english]{babel}

\usepackage[utf8]{inputenc} % allow utf-8 input
\usepackage[T1]{fontenc}    % use 8-bit T1 fonts
\usepackage{hyperref}       % hyperlinks
\usepackage{url}            % simple URL typesetting
\usepackage{booktabs}       % professional-quality tables
\usepackage{amsfonts}       % blackboard math symbols
\usepackage{nicefrac}       % compact symbols for 1/2, etc.
\usepackage{microtype} % microtypography

\usepackage{csquotes}

\usepackage{amsmath,amssymb,amsthm,bbm}
\usepackage{mathtools}
\usepackage{bbm}

\usepackage{xy}
\input xy \xyoption{all}

\usepackage{makecell}

\usepackage{cleveref}

\usepackage[dvipsnames]{xcolor}

\definecolor{veryperi}{RGB}{102,103,171}

\usepackage{tikz}    
\usepackage{pgfplots}
\pgfplotsset{compat=1.18}
\usepackage{mathrsfs}  

\usepackage{enumerate}\usepackage[inline]{enumitem}

\usepackage{caption}

\newcommand{\N}{\mathbb{N}}

\newcommand{\R}{\mathbb{R}}
\newcommand{\C}{\mathbb{C}}

\newcommand{\calA}{\mathcal{A}}

\newcommand{\bbR}{\mathbb{R}}

\newtheorem{thm}{Theorem}[section]

\newtheorem{prop}[thm]{Proposition}

\theoremstyle{definition}
\newtheorem{dfn}[thm]{Definition}
\newtheorem{ex}[thm]{Example}

\theoremstyle{remark}
\newtheorem{rmk}[thm]{Remark}

\numberwithin{equation}{section}

\begin{document}

\title{Unique wavelet sign retrieval from samples without bandlimiting}

%%    Information for first author
\author{Rima Alaifari}
%%    Address of record for the research reported here
\address{Seminar for Applied Mathematics, ETH Z{\"u}rich, R{\"a}mistrasse 101, 8092 Z{\"u}rich, Switzerland}
%%    Current address
%%\curraddr{Department of Mathematics and Statistics,
%%Case Western Reserve University, Cleveland, Ohio 43403}
\email{rima.alaifari@sam.math.ethz.ch}
%%    \thanks will become a 1st page footnote.
\thanks{The authors were supported in part by SNSF Grant
200021 184698.}
%
%%    Information for second author
\author{Francesca Bartolucci}
\address{Delft Institute of Applied Mathematics, 
TU Delft, Mekelweg 4, 2628 CD Delft, 
The Netherlands}
\email{f.bartolucci@tudelft.nl}
%%\thanks{Support information for the second author.}
%
\author{Matthias Wellershoff}
\address{Department of Mathematics, University of Maryland, 4176 Campus Drive, College Park, MD 20742, USA}
\email{wellersm@umd.edu}

%    General info

\date{\today}



\begin{abstract}
We study the problem of recovering a signal from magnitudes of its wavelet frame coefficients when the analyzing wavelet is real-valued. We show that every real-valued signal can be uniquely recovered, up to global sign, from its multi-wavelet frame coefficients  
\[
\{\lvert \mathcal{W}_{\phi_i} f(\alpha^{m}\beta n,\alpha^{m}) \rvert: i\in\{1,2,3\}, m,n\in\mathbb{Z}\}
\]
for every $\alpha>1,\beta>0$ with $\beta\ln(\alpha)\leq 4\pi/(1+4p)$, $p>0$, when the three wavelets $\phi_i$ are suitable linear combinations of the Poisson wavelet $P_p$ of order $p$ and its Hilbert transform $\mathscr{H}P_p$. For complex-valued signals we find that this is not possible for any choice of the parameters $\alpha>1,\beta>0$ and for any window.  In contrast to the existing literature on wavelet sign retrieval, our uniqueness results do not require any bandlimiting constraints or other a priori knowledge on the real-valued signals to guarantee their unique recovery from the absolute values of their wavelet coefficients.
\end{abstract}

\maketitle

\section{Introduction}
{\it Wavelet phase retrieval} refers to the inverse problem of reconstructing a square-integrable function $f$ from its {\it scalogram}; that is, from the absolute value of its {\it wavelet transform}:
\[
    \mathcal{W}_{\phi}f(b,a) := a^{-\frac{1}{2}} \int_{\R} f(x) \overline{\phi\left(\frac{x-b}{a}\right)} \,\mathrm{d} x, \qquad b \in \R,~a \in \R_+. 
\]
The wavelet transform emerged from research activities aimed at developing new analysis and processing tools to enhance signal theory and has proved to be extremely efficient in various applications such as denoising and compression. We refer to \cite{daub92,hol95,mallat98} for a thorough overview of wavelet analysis. However, there is still limited knowledge of the problem of reconstructing a function from the absolute value of its wavelet transform. This inverse problem arises in audio analysis and processing and has recently received an increasing amount of attention \cite{alaifari2017reconstructing,holighaus2019characterization,jaming2014uniqueness,mallat2015phase,waldspurger17}.

More precisely, wavelet phase retrieval aims at determining for which analyzing wavelets $\phi$, and which choices of sets $\Lambda \subseteq \R \times \R_+$ and subspaces $\mathcal{M} \subseteq L^2(\R)$ the forward operator 
\begin{equation}\label{eq:phaseretrievaloperator}
    \mathcal{A}_\phi : \mathcal{M} /\!\sim \, \to\, [0,+\infty)^\Lambda,\qquad \mathcal{A}_\phi (f)(b,a) := \lvert \mathcal{W}_\phi f (b,a) \rvert, \quad (b,a) \in \Lambda,
\end{equation}
is injective, where $f\sim g$ if and only if $f=\text{e}^{i\alpha}g$ for some $\alpha\in\R$. In the following, we distinguish between \emph{continuous} wavelet phase retrieval, meaning the recovery of $f$ (up to a global phase) from $\calA_\phi(f)$ when $\Lambda$ has the cardinality of the continuum, and \emph{sampled} wavelet phase retrieval, i.e.~the recovery of $f$ (up
to a global phase) from $\calA_\phi$ when $\Lambda$ is a discrete subset of $\R \times \R_+$.

We emphasise that in this paper we treat \emph{real-valued} analyzing wavelets and point out that whenever $\phi$ is real-valued the map 
\[
L^2(\R)/\!\sim\, \ni f \mapsto (|\mathcal{W}_\phi f(b,a)|)_{(b,a)\in\R\times\R_+}
\]
is not injective. To see that, we consider a function $f\in L^2(\R)$ with $\operatorname{Re}f\not\equiv 0$ and $\operatorname{Im}f\not\equiv 0$. Then, the functions $f$ and $g=\operatorname{Re}f-i\operatorname{Im}f$ are not equal up to a global constant phase, i.e.~$f \not\sim g$, but satisfy 
\[
|\mathcal{W}_\phi f(b,a)|=|\mathcal{W}_\phi g(b,a)|,\quad (b,a)\in\R\times\R_+.
\]
Therefore, we cannot hope to have a uniqueness result in $L^2(\R)/\!\sim$ when the analyzing wavelet is real-valued and the restriction to the space of real-valued square-integrable functions $L^2(\R,\R)/\!\sim$ is optimal. In this latter case, $f\sim g$ if and only if $f=\pm g$, and we refer to this problem as {\it wavelet sign retrieval}.

\subsection{Prior work: sampled wavelet sign retrieval with bandlimiting}\label{sec:priorwork}
The existing literature on sampled wavelet sign retrieval only includes uniqueness results which require either the analyzing wavelets or the signals to be bandlimited. In \cite{alaifari2017reconstructing} the authors show that every real-valued function $f\in L^2(\R)$ that has exponential decay at infinity is uniquely determined (up to a global sign) by its wavelet coefficients 
\[
\{\lvert \mathcal{W}_\phi f(2^{-m}\beta n,2^{-m}) \rvert: m\in\mathbb{N}, n\in\mathbb{Z}\}
\]
if $\phi$ is a real-valued bandlimited wavelet and $\beta>0$ is a sampling parameter explicitly determined by the bandwidth of $\phi$. Examples of real-valued bandlimited wavelets are the Meyer wavelet and the Shannon wavelet.

Then, in our recent paper \cite{baralawel2021}, we prove that, for every choice of the sampling parameters $\alpha > 1$, $\beta > 0$, and for every wavelet $\phi \in L^2(\R)$ with finitely many vanishing moments\footnote{We say that a wavelet has a \emph{finite number of vanishing moments} if there exists an $\ell\in\N$ such that
\begin{equation*}
    \lim_{\xi\to0}\xi^{-\ell}\widehat{\phi}(\xi)\in \C\setminus\{0\}.
\end{equation*}}, all real-valued bandlimited functions $f \in L^2(\R)$ are uniquely determined (up to a global sign) by the measurements 
\[
\{\lvert \mathcal{W}_\phi f(\alpha^{-m}\beta n,\alpha^{-m}) \rvert: m\in\mathbb{N}, n\in\mathbb{Z}\}.
\]
This uniqueness result can be restated as the injectivity of the operator $\mathcal{A}_\phi$, cf. \eqref{eq:phaseretrievaloperator}, for the choices 
\[
\mathcal{M}=\{f \in L^2(\R) : \operatorname{supp} \widehat{f}\ \mbox{ is compact}\}\qquad\text{and}\qquad \Lambda=\alpha^{-\N}(\beta \mathbb{Z} \times \{1\}),
\]
whenever $\phi$ has a finite number of vanishing moments. Examples of real-valued (non-bandlimited) wavelets with a finite number of vanishing moments include the Poisson wavelets and the $n^{{\textrm  {th}}}$ Hermitian wavelet, i.e.~the $n^{{\textrm  {th}}}$ derivative of the Gaussian function, for every $n\in\mathbb{N}_{>0}$. In particular, $n=2$ corresponds to the Mexican hat wavelet. 

\subsection{Our contribution: wavelet sign retrieval without bandlimiting}
The above-mentioned results always make a bandlimitedness assumption either on the unknown signals or on the analyzing wavelets. In contrast to prior work, our results here guarantee the unique recovery of real-valued signals from the absolute value of their wavelet transform without bandlimiting constraints or other a priori knowledge on the signals. While uniqueness results on sign retrieval from samples for $\mathcal{M}=L^2(\R,\R)$ are missing to date, the current paper proposes that a positive uniqueness result can be established when instead of a single wavelet, \emph{three wavelets} are employed.

We first prove Theorem~\ref{thm:signretrievalanalytic} which guarantees injectivity of $\mathcal{A}_\phi$ if $\Lambda$ is a subset of the upper half plane with positive measure and $\phi$ is such that $\operatorname{Ran}\mathcal{W}_{\phi}\subseteq\mathscr{A}(\C^+)$, where $\mathscr{A}(\C^+)$ denotes the space of real-analytic functions on the upper-half plane. Our theorem applies to the Poisson wavelet, a classical real-valued wavelet \cite[Chapter~1,\S~7]{holschneider91}.

However, Theorem~\ref{thm:signretrievalanalytic} does not apply to the more challenging case where $\Lambda$ is a discrete set, and we address this problem in  Theorem~\ref{thm:multiwavelets} with a {\it multi-wavelet approach}. A {\it hyperbolic lattice} in the upper half plane is a countable and discrete set of the form 
\[
\Lambda(\beta,\alpha) := \{(\alpha^m\beta n, \alpha^m)\}_{m,n\in\mathbb{Z}},
\]
with $\alpha>1$ and $\beta>0$, see Figure~\ref{fig:fandg1} for an example.
Theorem~\ref{thm:multiwavelets} shows that the magnitudes of the multi-wavelet frame coefficients
\begin{equation}
\label{eq:threemagnitudemeasurements}  
\{\lvert \mathcal{W}_{\phi_i} f(b,a) \rvert: i\in\{1,2,3\}, (b,a)\in\Lambda(\beta,\alpha)\}
% \{|\mathcal{W}_{\phi_1}f(b,a)|\}_{(b,a)\in\Lambda(\beta,\alpha)}\cup \{|\mathcal{W}_{\phi_2}f(b,a)|\}_{(b,a)\in\Lambda(\beta,\alpha)}\cup \{|\mathcal{W}_{\phi_3}f(b,a)|\}_{(b,a)\in\Lambda(\beta,\alpha)}
\end{equation}
uniquely determine $f$ up to a global sign if the analytic windows $\phi_1$, $\phi_2$, $\phi_3$ are suitable linear combinations of the \emph{Poisson wavelet}, defined in the Fourier domain by
\begin{equation*}
    \widehat{P}_p(\xi) := (2 \pi)^{p+1/2} |\xi|^p e^{-2 \pi |\xi|},\qquad\xi\in\mathbb{R},\quad p>0,
\end{equation*}
and its Hilbert transform $\mathscr{H}P_p$, and if $\Lambda$ is a hyperbolic lattice satisfying $\beta\ln(\alpha)\leq4\pi/(1+4p)$. For instance, we can choose $\phi_1=P_p$, $\phi_2=\mathscr{H}P_p$ and $\phi_3=P_p+\mathscr{H}P_p$.

As a consequence of Proposition~\ref{prop:Poissonframe}, the set
\begin{equation}\label{eq:waveletframe}
    \{T_bD_aP_p\}_{(b,a)\in\Lambda(\beta,\alpha)} \cup \{T_bD_a\mathscr{H}P_p\}_{(b,a)\in\Lambda(\beta,\alpha)}
\end{equation}
forms a frame for $L^2(\R,\R)$ if and only if $\beta\ln(\alpha)<2\pi/p$, where $T_b$ and $D_a$ denote the translation and dilation operators 
\begin{equation*}
    T_bf(x)=f(x-b) \mbox{ and } D_af(x)=a^{-\frac{1}{2}}f\left(a^{-1}x\right),
\end{equation*}
respectively. Theorem~\ref{thm:multiwavelets} states that by taking $(1+4p)/2p$ times the necessary density for \eqref{eq:waveletframe} to constitute a wavelet frame for $L^2(\R,\R)$, we can achieve uniqueness in the recovery of real-valued signals from the set of measurements \eqref{eq:threemagnitudemeasurements}. Therefore, we need roughly $3(1+4p)/4p$ 
times more samples than required for signal reconstruction using wavelet coefficients when the phases are not available. For example, if $p=1$, we would need $3.75$ times more samples to obtain uniqueness. It is still an open question whether three wavelets are indeed necessary or whether one, or two, of the analytic wavelets $P_p$ and $\mathscr{H}P_p$ (or a linear combination of the two) could also be sufficient to obtain uniqueness.

\subsection{Other related literature.}
In \cite{mallat2015phase} the authors prove that the magnitude of the Cauchy wavelet transform of a signal $f \in L^2(\R)$ uniquely determines its \emph{analytic representation}
\begin{equation*}
    \widehat f_+ (\xi) := 2 \widehat f (\xi) \boldsymbol{1}_{\xi > 0}, \qquad \xi \in \R,
\end{equation*}
up to a global constant phase factor. Here, $\boldsymbol{1}_\Xi$ denotes the characteristic function of the set $\Xi$. Since real-valued functions are uniquely determined by their analytic representation, one would be tempted to conclude that the magnitude of the Cauchy wavelet transform uniquely determines real-valued signals up to a global sign. However, the analytic representation $f_+$ up to a global phase factor does not uniquely determine the real-valued signal up to a global sign. Indeed, it is possible to construct real-valued functions $f,g \in L^2(\R)$ which do not agree up to a global sign but satisfy $g_+ = e^{i\alpha} f_+$, for some $\alpha \in \R$, see \cite[Remark~12]{baralawel2021}. In contrast with \cite{mallat2015phase}, our uniqueness theorem guarantees the unique recovery of the real-valued signals themselves, instead of their analytic representations. Moreover, the result in \cite{mallat2015phase} is a uniqueness result in the semi-discrete regime, while our main result is in the fully discrete setting. 

Building on the results in \cite{mallat2015phase}, in \cite{baralawel2021}
we restrict the signal class to analytic bandlimited signals to obtain a uniqueness result from sampled Cauchy
wavelet transform measurements. Again, the bandlimitedness assumption plays a crucial role to obtain a full sampling result. 

Finally, we point out that our result is reminiscent of the recent work in \cite{grohs2022multi}, where the authors prove uniqueness results from magnitudes of {\em multi-window} Gabor frame coefficients. In particular, the argument in the proof of Theorem~\ref{thm:multiwavelets} to obtain equation~\eqref{eq:reference} has been inspired by the results in \cite{grohs2022multi}.

\subsection*{Notation.} We set $\R_+ := (0,+\infty)$. For any $p\in[1,+\infty]$, we denote by $L^p(\R)$ the Banach space of functions $f:\R\rightarrow\C$ which are $p$-integrable with respect to the Lebesgue measure and we use the notation $\|\cdot\|_p$ for the corresponding norms. The \emph{Fourier transform} on $L^1(\R)$ is defined by
\[
\widehat{f}(\xi) := \int_{\R} f(x) {\rm e}^{-2\pi \mathrm{i} 
  x \xi } \,\mathrm{d} x, \qquad \xi \in \R,
\] 
and it extends to $L^2(\R)$ by a classical density argument. Finally, we denote by $\xi\mapsto\text{sgn}(\xi)$ the sign function.

\section{Preliminaries}
\subsection{Frames in Hilbert spaces.} 

The notion of a frame, which generalizes that of a Riesz basis in Hilbert spaces, is due to Duffin and Schaffer \cite{duffin1952class}.  Let $\mathcal{H}$ be a separable Hilbert space with inner product $\langle \cdot,\cdot\rangle$ and norm $\|\cdot\|$. A countable sequence of vectors $\{f_i\}_{i\in I}$ in $\mathcal{H}$ is a {\it frame} for $\mathcal{H}$ if there exist constants $A,B>0$ such that for all $f\in\mathcal{H}$
\[
A\|f\|^2\leq\sum_{i\in I}|\langle f,f_i\rangle|^2\leq B \|f\|^2.
\]
A direct consequence of the lower inequality is that every $f\in\mathcal{H}$ is uniquely determined by its {\it frame coefficients} $\{\langle f,f_i\rangle\}_{i\in I}$, that is 
\[
\langle f,f_i\rangle=0,\ \forall i\in I \implies f=0.
\]
If a frame is obtained via translated and scaled versions $\{T_bD_a\phi\}_{(b,a)\in I}$ of a fixed function $\phi$, the analyzing wavelet, we call it a {\it wavelet frame}. A generalization of wavelet frames is obtained by considering $k\in\mathbb{N}$, $k>1$, analyzing wavelets instead of a single one. In this latter case, we talk about {\it multi-wavelet frames}. We refer to \cite{christensen2003introduction, daub92} for an introduction to frame theory and an overview on wavelet frames. 

\subsection{Weighted Bergman spaces}
We denote by $\mathbb{C}^+$ the upper half plane,
\[
\mathbb{C}^+=\{z\in\C : \text{Im}\,z>0\},
\]
and by $\mathcal{O}(\mathbb{C}^+)$ the space of holomorphic functions on $\mathbb{C}^+$.
For every $w>1$, we define the {\it weighted Bergman space} $B_{w}(\mathbb{C}^+)$ by
\[
B_w(\mathbb{C}^+) := \left\{F\in\mathcal{O}(\mathbb{C}^+) : \int_{\R\times\R_+}|F(x+iy)|^2y^{w-2}\,{\rm d}x\,{\rm d}y<\infty\right\}.
\]
We say that a discrete subset $\Lambda$ of $\mathbb{C}^+$ is a {\it set of sampling} for $B_{w}(\mathbb{C}^+)$ if there exist positive constants $A$ and $B$ such that
\begin{multline*}
A \int_{\R\times\R_+}|F(x+iy)|^2y^{w-2}{\rm d}x{\rm d}y\leq\sum_{z_j=x_j+iy_j\in\Lambda}|F(x_j+iy_j)|^2y_j^w \\ \leq B \int_{\R\times\R_+}|F(x+iy)|^2\,y^{w-2}\,{\rm d}x\,{\rm d}y,
\end{multline*}
for every $F\in B_{w}(\mathbb{C}^+)$. 

\begin{thm}[{\cite[Theorem~1.1]{seip1993regular}}]\label{thm:seip1}
Let $w>1$. For every $\alpha>1$ and $\beta>0$, the discrete set
\[
\Gamma(\beta,\alpha)=\{\alpha^m(\beta n+i)\}_{m,n\in\mathbb{Z}}
\]
is a set of sampling for $B_{w}(\mathbb{C}^+)$ if and only if 
\[
\beta\ln(\alpha)<4\pi/(w-1).
\]
\end{thm} 
Furthermore, we say that a discrete subset $\Lambda$ of $\mathbb{C}^+$ is a {\it uniqueness set} for $B_{w}(\mathbb{C}^+)$ if for every $F\in B_{w}(\mathbb{C}^+)$ it holds that
\[
F(\lambda)=0, \ \forall\lambda\in\Lambda \implies F\equiv0.
\]
Clearly, sampling sets are also uniqueness sets.
\
\begin{thm}[\cite{seip1993regular}]\label{thm:seip}
Let $w>1$. For every choice of $\alpha>1$ and $\beta>0$ satisfying 
\[
\beta\ln(\alpha)=4\pi/(w-1),
\]
the discrete set
\[
\Gamma(\beta,\alpha)=\{\alpha^m(\beta n+i)\}_{m,n\in\mathbb{Z}}
\]
is a uniqueness set for $B_{w}(\mathbb{C}^+)$.
\end{thm} 
A general treatment of sets of sampling for weighted Bergman spaces can be found in \cite{seip199nonregular}.
\begin{figure}
    \centering
\includegraphics[scale=0.6]{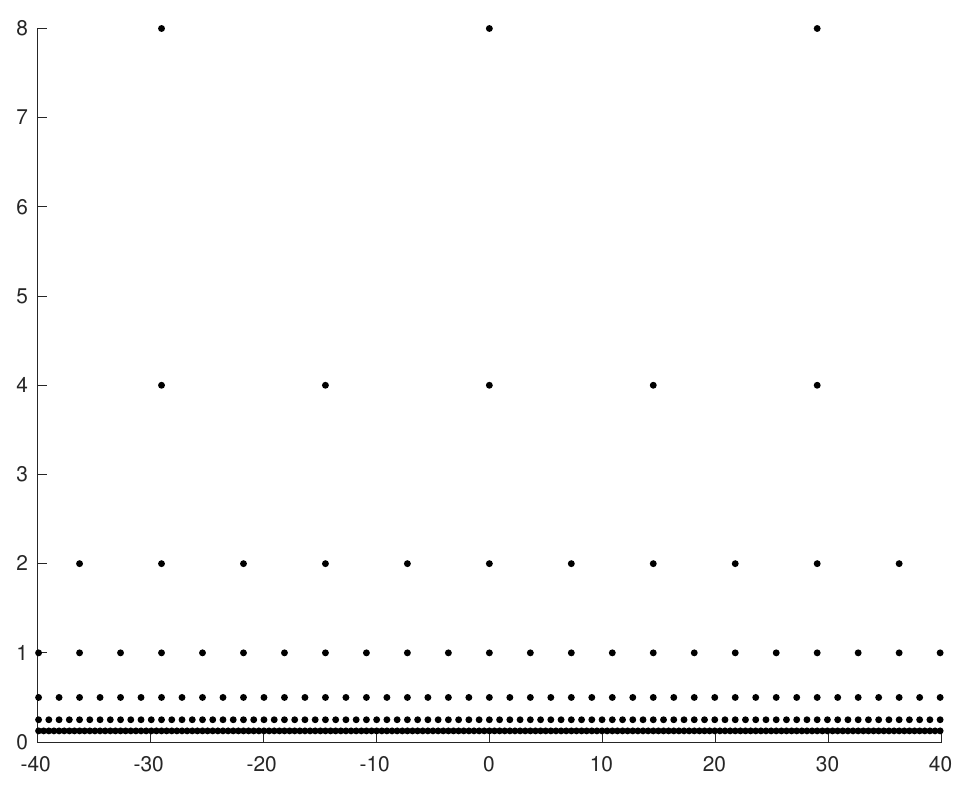}
    \caption{The lattice $\Lambda(\beta,\alpha)$ with $\beta=4\pi/5\ln(2)$ and  $\alpha=2$ constitutes a uniqueness set for the weighted Bergman space $B_{6}(\mathbb{C}^+)$.}
    \label{fig:fandg1}
\end{figure}

The weighted Bergman spaces appear in the characterization of the range of the wavelet transform for a special choice of the analyzing wavelets: the so-called Cauchy wavelets.

 \subsection{Cauchy wavelets} 
 The {\it Cauchy wavelets} are progressive wavelets, i.e. wavelets with only positive frequencies, with real-valued Fourier transforms defined by
\[
    \widehat{\psi}_p(\xi):=(2\pi)^{p + 1/2}\xi^p\text{e}^{-2\pi \xi}\boldsymbol{1}_{\xi>0},\quad\xi\in\R,
\]
where $p>0$. We abbreviate the case $p=1$ by $\psi=\psi_1$. The wavelet transform with respect to any Cauchy wavelet is referred to as the {\it Cauchy wavelet transform} and is defined by 
\[
    \mathcal{W}_{\psi_p}f(b,a):=(2\pi a)^{\frac{1}{2}+p}\int_{0}^{+\infty}\widehat{f}(\xi)\,\text{e}^{2\pi i (b+ i a)\xi}\,\xi^p\,{\rm d}\xi,\quad (b,a)\in\R\times\R_+.
\]
The Cauchy wavelets satisfy the admissibility condition  
\[
C_{\psi_p}=\int_{0}^{+\infty}\frac{|\widehat{\psi}_p(\xi)|^2}{\xi}{\rm d}\xi<\infty,
\]
and consequently the Cauchy wavelet transform is a multiple of an isometry from $\mathcal{H}_+(\R)$ into $L^2(\R\times\R_+,{\rm d}b\,{\rm d}a/a^2)$,
\begin{equation}\label{eq:waveletisometry}
\|\mathcal{W}_{\psi_p}f\|_{L^2(\R\times\R_+,{\rm d}b\,{\rm d}a/a^2)}^2=C_{\psi_p}\|f\|_2^2,
\end{equation}
where $\mathcal{H_+}(\R)$ denotes the \emph{Hardy space},
\[
\mathcal{H}_+(\R):=\{f\in L^2(\R) : \forall \xi < 0,~\widehat{f}(\xi)=0\}.
\]
The weighted Bergman spaces relate to the Cauchy wavelet transform as the Bargmann--Fock space does to the Gabor transform. Indeed, the map 
\begin{equation}\label{eq:mapbergmanwavelet}
    \mathcal{H}_+(\R)\ni f\mapsto F\colon{\mathbb{C}^+}\to\mathbb{C},\quad F(b+ia)=a^{-(\frac{1}{2}+p)}\mathcal{W}_{\psi_p}f(b,a),    
\end{equation}
defines a multiple of an isometry from $\mathcal{H}_+(\R)$ onto the weighted Bergman space $B_{2p+1}(\mathbb{C}^+)$,
\[
\int_{\R\times\R_+} |F(b+ia)|^2 a^{2p-1}\,{\rm d}b\,{\rm d}a=C_{\psi_p}\|f\|_2^2,
\]
see \cite[\S~3.2]{grossmann1986transforms} and the references therein.

\subsection{Poisson wavelets} 
The {\it Poisson wavelet} is a real-valued wavelet with real-valued Fourier transform given by
\begin{equation}\label{eq:PoissonFourier}
    \widehat{P}(\xi)=\sqrt{2\pi} |\xi|e^{-2\pi|\xi|},\qquad\xi\in\mathbb{R},
\end{equation}
see e.g.~\cite[Chapter~1,\S~7]{holschneider91}.
Equation~\eqref{eq:PoissonFourier} implies that the Poisson wavelet is related to the Cauchy wavelet via 
\begin{equation*}
2\psi= P_+=P+i\mathscr{H}P,
\end{equation*}
where $\mathscr{H}$ denotes the Hilbert transform. We recall that the Hilbert transform is a unitary operator on $L^2(\R)$ defined by 
\begin{equation*}
(\mathscr{H}\varphi)\;\widehat{}\;(\xi) = -i\, \text{sgn}(\xi)\widehat{\varphi}(\xi),\quad \text{for a.e. } \xi\in\R.
\end{equation*}
In particular, the Hilbert transform of a real-valued function remains real-valued. 
%\begin{equation*}
%    \widehat{P}(\omega)=\widehat{\psi}(\omega)+\widehat{\psi}(-\omega),\quad\omega\in\R.
%\end{equation*}
In general, for every $p>0$, we can define the {\it Poisson wavelet of order $p$} as
%\begin{equation}\label{eq:poissonwavelets}
%    \widehat{P}_p(\omega)=\widehat{\psi_p}(\omega)+\widehat{\psi_p}(-\omega)=|\omega|^p e^{-|\omega|},\qquad\omega\in\mathbb{R}.
%\end{equation}
\begin{equation*}
    \widehat{P}_p(\xi):=
    %\widehat{\psi_p}(\omega)+\widehat{\psi_p}(-\omega)=
    (2\pi)^{p+1/2}|\xi|^p e^{-2\pi|\xi|},\qquad\xi\in\mathbb{R},
\end{equation*}
which is equivalent to
\begin{equation}\label{eq:poissonwavelets}
2\psi_p=P_p+i\mathscr{H}P_p.
\end{equation}

 The next proposition shows that the wavelet transform of any real-valued signal with respect to a Poisson wavelet corresponds to the real part of its Cauchy wavelet transform. Analogously, the wavelet transform of any real-valued signal with respect to the Hilbert transform of a Poisson wavelet relates to the imaginary part of its Cauchy wavelet transform. 
 
\begin{prop}\label{prop:cauchyandpoisson}
    Let $p>0$. For every $f\in L^2(\R,\R)$,
    \begin{align}\label{eq:wavelettransformcauchypoisson}
    \mathcal{W}_{P_p}f&=2\operatorname{Re}[\mathcal{W}_{\psi_p}f],\quad \mathcal{W}_{\mathscr{H}P_p}f=-2\operatorname{Im}[\mathcal{W}_{\psi_p}f].
    \end{align}
    In particular, $\mathcal{W}_{P_p}f$ and $\mathcal{W}_{\mathscr{H}P_p}f$ are real analytic functions on the upper half-plane.
\end{prop}

\begin{proof}
    Let $p>0$ and $f\in L^2(\R,\R)$. By equation~\eqref{eq:poissonwavelets}, and the definition of the wavelet transform, we obtain
    \begin{align*}
        2 \mathcal{W}_{\psi_p}f(b,a)=2 \langle f, T_bD_a\psi_p\rangle&=\langle f, T_bD_a\left(P_p+i\mathscr{H}P_p\right)\rangle,\\
        &=\langle f, T_bD_aP_p\rangle-i\langle f,T_bD_a\mathscr{H}P_p\rangle,\\
        &=\mathcal{W}_{P_p}f(b,a)-i\mathcal{W}_{\mathscr{H}P_p}f(b,a),
    \end{align*}
    for $(b,a)\in\R\times\R_+$, which implies equation~\eqref{eq:wavelettransformcauchypoisson}.
    Equivalently, for every $(b,a)\in\R\times\R_+,$ we have that
    \begin{align*}
        &\mathcal{W}_{P_p}f(b,a)=2a^{\frac{1}{2}+p}\text{Re}[a^{-(\frac{1}{2}+p)}\mathcal{W}_{\psi_p}f(b,a)],\\
        &\mathcal{W}_{\mathscr{H}P_p}f(b,a)=-2a^{\frac{1}{2}+p}\text{Im}[a^{-(\frac{1}{2}+p)}\mathcal{W}_{\psi_p}f(b,a)],
    \end{align*}
    which, together with equation~\eqref{eq:mapbergmanwavelet}, shows that $\mathcal{W}_{P_p}f$ and $\mathcal{W}_{\mathscr{H}P_p}f$ are real analytic functions on the upper half-plane.
\end{proof}

As a consequence of Theorem~\ref{thm:seip1} and Proposition~\ref{prop:cauchyandpoisson}, the Poisson wavelet and its Hilbert transform  give rise to discrete frames for $L^2(\R,\R)$. More precisely, we have the following result.

\begin{prop}\label{prop:Poissonframe}
    The set
    \begin{equation}\label{eq:waveletframe2}
    \{T_bD_aP_p\}_{(b,a)\in\Lambda(\beta,\alpha)}\cup \{
    T_bD_a\mathscr{H}P_p\}_{(b,a)\in\Lambda(\beta,\alpha)}
    \end{equation}
    constitutes a frame for $L^2(\R,\R)$ if and only if 
    $\beta\ln(\alpha)<2\pi/p$.
\end{prop}

\begin{proof}
    The set in equation~\eqref{eq:waveletframe2} is a frame for $L^2(\R,\R)$ if and only if there exist constants $A,B>0$ such that 
    \begin{equation}\label{eq:frameequivalence1}
        A \|f\|_2^2\leq\sum_{(b,a)\in \Lambda(\beta,\alpha)}(|\mathcal{W}_{P_p}f(b,a)|^2+|\mathcal{W}_{\mathscr{H}P_p}f(b,a)|^2)\leq B \|f\|_2^2,
    \end{equation}
    for every $f\in L^2(\R,\R)$.
    By Proposition~\ref{prop:cauchyandpoisson}, equation \eqref{eq:frameequivalence1} is equivalent to
    \begin{equation}\label{eq:frameequivalence2}
        \frac{A}{4} \cdot \|f\|_2^2\leq\sum_{(b,a)\in \Lambda(\beta,\alpha)}|\mathcal{W}_{\psi_p}f(b,a)|^2\leq \frac{B}{4} \cdot \|f\|_2^2, \mbox{ for all } f\in L^2(\R,\R).
    \end{equation}
    We recall that any real-valued function satisfies 
    \[
    \widehat{f}(-\xi)=\overline{\widehat{f}(\xi)},\quad\xi\in\R.
    \]
    Consequently, $L^2(\R,\R)$ is isomorphic to the Hardy space $\mathcal{H}_+(\R)$ via the mapping
    \begin{equation*}
    L^2(\R,\R)\ni f\mapsto f_+\in \mathcal{H}_+(\R),
    \end{equation*}
    which satisfies
    \begin{equation*}
       \|f\|_2=\frac{ \|f_+\|_2}{\sqrt{2}}.
    \end{equation*}
    %where we recall that $f_+$ denotes the analytic representation of $f$
    %\begin{equation*}
    %    \widehat f_+ (\xi) = 2 \widehat f (\xi) \boldsymbol{1}_{\xi > 0}, \qquad \xi \in \R.
    %\end{equation*}
    
    Furthermore, we can see that the wavelet transform with respect to a Cauchy wavelet satisfies 
    \begin{align*}
        \mathcal{W}_{\psi_p} f (b,a) &= \sqrt{a} \cdot \int_\R \widehat f (\xi) \overline{\widehat \psi_p (a\xi)} \mathrm{e}^{2\pi\mathrm{i}\xi b} \, \mathrm{d} \xi = \sqrt{a} \cdot \int_0^\infty \widehat f (\xi) \overline{\widehat \psi_p (a\xi)} \mathrm{e}^{2\pi\mathrm{i}\xi b} \, \mathrm{d} \xi \\ &= \frac{\sqrt{a}}{2} \cdot \int_0^\infty \widehat f_+ (\xi) \overline{\widehat \psi_p (a\xi)} \mathrm{e}^{2\pi\mathrm{i}\xi b} \, \mathrm{d} \xi = \frac12 \cdot \mathcal{W}_{\psi_p} f_+ (b,a),
    \end{align*}
    where the first equality follows by Plancherel's theorem. Therefore, equation~\eqref{eq:frameequivalence2} is equivalent to
    \begin{equation}\label{eq:frameequivalence3}
        \frac{A}{2} \cdot \|f\|_2^2\leq\sum_{(b,a)\in \Lambda(\beta,\alpha)}|\mathcal{W}_{\psi_p}f(b,a)|^2\leq \frac{B}{2} \cdot \|f\|_2^2,\mbox{ for all } f\in \mathcal{H}_+(\R).
    \end{equation}
    Now, since the map 
    \begin{equation*}
        \mathcal{H}_+(\R)\ni f\mapsto F\colon{\mathbb{C}^+}\to\mathbb{C},\quad F(b+ia)=a^{-(\frac{1}{2}+p)}\mathcal{W}_{\psi_p}f(b,a),    
    \end{equation*}
    is a multiple of an isometry from $\mathcal{H}_+(\R)$ onto the weighted Bergman space $B_{2p+1}(\mathbb{C}^+)$ (cf.~equation~\eqref{eq:mapbergmanwavelet}), equation~\eqref{eq:frameequivalence3} becomes
    \begin{multline}\label{eq:frameequivalence4}
        \frac{A}{2 C_{\psi_p}} \cdot \int_{\R\times\R_+}|F(x+iy)|^2y^{2p-1}{\rm d}x{\rm d}y \leq \sum_{x+iy\in \Gamma(\beta,\alpha)}|F(x+iy)|^2y^{2p+1} \\ 
        \leq \frac{B}{2 C_{\psi_p}} \cdot \int_{\R\times\R_+}|F(x+iy)|^2y^{2p-1}{\rm d}x{\rm d}y,
    \end{multline}
    for every $F\in B_{2p+1}(\mathbb{C}^+)$, where $\Gamma(\beta,\alpha)$ denotes the image of $\Lambda(\beta,\alpha)$ via the isomorphism $\R^2\ni(x,y)\mapsto x+iy\in\mathbb{C}$. Finally, Theorem~\ref{thm:seip1} allows us to conclude that equation~\eqref{eq:frameequivalence4} is satisfied if and only if $\beta\ln(\alpha)<2\pi/p$.
\end{proof}

With this preparatory result, we now work towards establishing our main result. For this, we say that a frame $\Phi=\{\varphi_{i}\}_{i\in I}$ of a separable Hilbert space $\mathcal{H}$ \emph{does phase retrieval} if the nonlinear map 
\begin{equation*}
  \mathcal{A}_{\Phi}\colon\mathcal{H}/\!\sim\,\to \R_+^I,\qquad  \mathcal{A}_{\Phi}([f])=\{|\langle f, \varphi_i\rangle|\}_{i\in I}  
\end{equation*}
is injective, where $f\sim g$ if and only if $f=\text{e}^{i\alpha}g$ for some $\alpha\in\R$. Building on Proposition~\ref{prop:Poissonframe}, our main theorem shows that the multi-wavelet frame 
\begin{equation*}
\{T_bD_a P_p\}_{(b,a)\in\Lambda(\beta,\alpha)}\cup \{T_bD_a \mathscr{H}P_p\}_{(b,a)\in\Lambda(\beta,\alpha)}\cup \{T_bD_a (\mu_1 P_p+\mu_2 \mathscr{H}P_p)\}_{(b,a)\in\Lambda(\beta,\alpha)}
\end{equation*}
does phase retrieval in $L^2(\R,\R)$ if $\Lambda(\beta,\alpha)$ is a hyperbolic lattice with $\beta\ln(\alpha)\leq4\pi/(1+4p)$, and if the collection of vectors 
\[
\{(1,0),(0,1),(\mu_1,\mu_2)\}
\]
satisfies the complement property whose definition is recalled in the next section.
 
\subsection{Sign retrieval in \texorpdfstring{$\mathbbm{R}^{\boldsymbol{M}}$}{RM}.}
Given a collection of vectors $\Phi=\{v_{n}\}_{n=1}^N$ in $\R^M$, we consider the map 
\[
\mathcal{A}_{\Phi}\colon\R^M/\!\sim\,\to \R_+^N,\qquad  \mathcal{A}_{\Phi}(v)=\{|\langle v, v_{n}\rangle|\}_{n=1}^N,
\]
where $v\sim w$ if and only if $v=\pm w$. The term ``sign retrieval in $\bbR^M$'' refers to the study of necessary and sufficient conditions on $\Phi$ under which the map $\mathcal{A}_{\Phi}$ is injective. It is well-known that the map $\mathcal{A}_{\Phi}$ is injective if and only if the collection $\Phi$ has the complement property.
\begin{dfn}
Let $M, N\in\N$ and let $\Phi=\{v_{n}\}_{n=1}^N$ be a collection of vectors in $\R^M$. We say that $\Phi$ has the {\it complement property} if, for every subset $S\subseteq\{1,\ldots,N\}$, either $\text{span}\{v_{n}: n\in S\}=\R^M$ or $\text{span}\{v_{n}: n\in \{1,\ldots,N\}\setminus S\}=\R^M$.
\end{dfn}
\begin{prop}[\cite{balan2006signal}]\label{prop:cp}
Let $M, N\in\N$ and let $\Phi=\{v_{n}\}_{n=1}^N$ be a collection of vectors in $\R^M$. Then, the map 
\begin{equation}\label{eq:mapsignretrieval}
  \mathcal{A}_{\Phi}\colon\R^M/\{\pm 1\}\to \R_+^N,\qquad  \mathcal{A}_{\Phi}(v)=\{|\langle v, v_{n}\rangle|\}_{n=1}^N,  
\end{equation}
is injective if and only if $\Phi$ has the complement property.
\end{prop}
Proposition~\ref{prop:cp} immediately implies that the map in equation~\eqref{eq:mapsignretrieval} is not injective if $N\leq 2M-2$. A natural question is therefore whether $N=2M-1$ vectors are sufficient to yield injectivity of $\mathcal{A}_\Phi$. This question has been answered in \cite{balan2006signal} where the authors show that a collection of vectors $\Phi=\{v_{n}\}_{n=1}^{2M-1}\subseteq \R^M$ has the complement property if and only if $\Phi$ is {\it full spark}, which means that every subcollection of $M$ vectors of $\Phi$ spans $\R^M$.

The complement property can be stated in the more general setting where $\Phi$ is a collection of vectors in a separable Hilbert space $\mathcal{H}$. 

\begin{dfn}[\cite{cahill2016phase}]\label{def:cphilbert}
Let $\Phi=\{\varphi_{i}\}_{i\in I}$ be a collection of vectors in a separable Hilbert space $\mathcal{H}$. We say that $\Phi$ has the \emph{complement property} if, for every subset $S\subseteq I$, either $\overline{\text{span}\{\varphi_{i}: i\in S\}}=\mathcal{H}$ or $\overline{\text{span}\{\varphi_{i}: i\in I\setminus S\}}=\mathcal{H}$.
\end{dfn}

Analogously to the finite dimensional setting, if $\mathcal{H}$ is a Hilbert space over $\R$, a collection $\Phi$ does phase retrieval if and only if $\Phi$ has the complement property. If $\mathcal{H}$ is a Hilbert space over $\C$, we only know  that the complement property is a necessary condition for uniqueness of the phase retrieval problem, see \cite{alaifari2017phase,cahill2016phase}.

\section{Main Results: Wavelet Sign Retrieval without Bandlimiting}
This section is devoted to presenting our contributions to wavelet sign retrieval. The novelty of our results with respect to the existing literature lies in the fact that we do not need any a priori knowledge about the real-valued square-integrable signals to guarantee their unique recovery from the absolute values of their wavelet transforms.

% \ern{(This paragraph has already appeared in the above..)} In wavelet sign retrieval both the analyzing wavelets and the signals are assumed to be real-valued. We point out that, whenever the wavelet $\phi$ is real-valued, the map 
% \[
% L^2(\R)\ni f\mapsto (|\mathcal{W}_\phi f(b,a)|)_{(b,a)\in\R\times\R_+}
% \]
% is not injective. To see that, we consider a function $f\in L^2(\R)$ with $\operatorname{Im}f\not\equiv 0$. Then, the functions $f$ and $g=\operatorname{Re}f-i\operatorname{Im}f$ are not equal up to a global constant phase but satisfy 
% \[
% |\mathcal{W}_\phi f(b,a)|=|\mathcal{W}_\phi g(b,a)|,\quad\forall(b,a)\in\R\times\R_+.
% \]
% Therefore, we cannot hope to have a uniqueness result in $L^2(\R)$ when the analyzing wavelet is real-valued, and the restriction to $L^2(\R,\R)$ is thus optimal.

Recall that we distinguish between continuous wavelet sign retrieval, meaning the recovery of $f$ (up to a global sign) from $(|\mathcal{W}_\phi f(b,a)|)_{(b,a)\in\Omega}$, where $\Omega$ has the cardinality of the continuum, and sampled wavelet sign retrieval, i.e.~the recovery of $f$ (up
to a global sign) from $(|\mathcal{W}_\phi f(b,a)|)_{(b,a)\in\Lambda}$, where $\Lambda$ is a discrete subset of $\R \times \R_+$.

\subsection{Continuous wavelet sign retrieval.}
We recall that $\mathscr{A}(\C^+)$ denotes the space of real-analytic functions on the upper-half plane. 
\begin{thm}\label{thm:signretrievalanalytic} 
Let $\Omega\subseteq \R\times\R_+$ be a set with positive measure and let $\phi\in L^2(\R,\R)$ be an analyzing wavelet such that $\operatorname{Ran}\mathcal{W}_{\phi}\subseteq\mathscr{A}(\C^+)$. Then, the following are equivalent for $f,g\in L^2(\R,\R)$:
\begin{itemize}
    \item[(i)] $f=\pm g$,
    \item[(ii)] $|\mathcal{W}_{\phi}f|=|\mathcal{W}_{\phi}g|$ on $\Omega$.
\end{itemize}
\end{thm}
\begin{proof}
It is clear that item (i) implies item (ii). Let us therefore assume that $f,g\in L^2(\R,\R)$ satisfy
\[
|\mathcal{W}_\phi f(b,a)|=|\mathcal{W}_\phi g(b,a)|, \qquad (b,a)\in\Omega.
\]
If we consider the subset
\begin{equation*}
    S=\{(b,a)\in\Omega:\mathcal{W}_\phi f(b,a)=\mathcal{W}_\phi g(b,a)\},
\end{equation*}
then the functions $h_1=f+g$ and $h_2=f-g$ satisfy
\[
    \mathcal{W}_{\phi}h_1(b,a)=0, \qquad (b,a)\in \Omega\setminus S,
\]
and 
\[
\mathcal{W}_{\phi}h_2(b,a)=0, \qquad (b,a)\in S.
\]
Since $\Omega$ has positive measure, either $S$ or $\Omega\setminus S$ has positive measure. Therefore, either $\mathcal{W}_{\phi}h_1\equiv0$ or $\mathcal{W}_{\phi}h_2\equiv0$ because the zero set of a non-zero real analytic function has zero measure. We can thus conclude that $f=\pm g$.
\end{proof}
\begin{rmk}
By Proposition~\ref{prop:cauchyandpoisson}, Theorem~\ref{thm:signretrievalanalytic} applies to the family of the Poisson wavelets and to their Hilbert transforms. 
% More in general, $\mathcal{W}_{\phi}\phi\in\mathcal{A}(\C^+)$ is a sufficient condition for $\operatorname{Ran}\mathcal{W}_{\phi}\subseteq\mathcal{A}(\C^+)$.  
\end{rmk}

\subsection{Sampled wavelet sign retrieval}
Theorem~\ref{thm:signretrievalanalytic} does not apply to the case where we only know the magnitude of the wavelet transform on a discrete set. We recall that a hyperbolic lattice in the upper half-plane takes the form 
\[
\Lambda(\beta,\alpha)=\{(\alpha^m\beta n,\alpha^m)\}_{m,n\in\mathbb{Z}}
\]
for some values of the parameters $\alpha>1$ and $\beta>0$. It is therefore a natural question to ask under which assumptions on the analyzing wavelets and on the sample parameters $\alpha$ and $\beta$ we can obtain a uniqueness result as in Theorem~\ref{thm:signretrievalanalytic}. We address this problem with a multi-wavelet approach. 

\begin{thm}\label{thm:multiwavelets}
Let $p>0$ and let $\alpha>1$ and $\beta>0$ satisfy
\[
\beta\ln(\alpha)\leq\frac{4\pi}{1+4p}. 
\]
Furthermore, let
\[
\phi_i=\lambda_{i,1}P_p+\lambda_{i,2}\mathscr{H}P_p,\quad i=1,2,3,
\]
for a collection of full spark vectors 
\[
\{\lambda_i=(\lambda_{i,1},\lambda_{i,2})\}_{i=1}^{3}
\]
in $\mathbb{R}^2.$ 
Then, every $f\in L^2(\R,\R)$ is uniquely determined, up to a global sign, by the magnitudes of the multi-wavelet frame coefficients
\begin{equation*} 
\{\lvert \mathcal{W}_{\phi_i} f(b,a) \rvert: i\in\{1,2,3\}, (b,a)\in\Lambda(\beta,\alpha)\}.
% \{|\mathcal{W}_{\phi_1}f(b,a)|\}_{(b,a)\in\Lambda(\beta,\alpha)}\cup \{|\mathcal{W}_{\phi_2}f(b,a)|\}_{(b,a)\in\Lambda(\beta,\alpha)}\cup \{|\mathcal{W}_{\phi_3}f(b,a)|\}_{(b,a)\in\Lambda(\beta,\alpha)}
\end{equation*} 
% magnitude-only measurements 
% \begin{equation*}
% \{|\mathcal{W}_{\phi_1}f(b,a)|\}_{(b,a)\in\Lambda(\beta,\alpha)}\cup \{|\mathcal{W}_{\phi_2}f(b,a)|\}_{(b,a)\in\Lambda(\beta,\alpha)}\cup \{|\mathcal{W}_{\phi_3}f(b,a)|\}_{(b,a)\in\Lambda(\beta,\alpha)}.
% \end{equation*}
% if the analyzing wavelets are of the form
% \[
% \phi_i=\lambda_{i,1}P_p+\lambda_{i,2}\mathscr{H}P_p,\quad i=1,2,3,
% \]
% for a collection of full spark vectors 
% \[
% \{\lambda_i=(\lambda_{i,1}+\lambda_{i,2})\}_{i=1}^{3}.
% \]
Equivalently, the multi-wavelet frame 
\begin{equation*}
\{T_bD_a\phi_1\}_{(b,a)\in\Lambda(\beta,\alpha)}\cup \{T_bD_a\phi_2\}_{(b,a)\in\Lambda(\beta,\alpha)}\cup \{T_bD_a\phi_3\}_{(b,a)\in\Lambda(\beta,\alpha)}.
\end{equation*}
does sign retrieval in $L^2(\R,\R)$.
\end{thm}
\begin{proof}
Let $f,g\in L^2(\R,\R)$ be such that 
\[
|\mathcal{W}_{\phi_i} f(b,a)|=|\mathcal{W}_{\phi_i} g(b,a)|,\quad (b,a)\in\Lambda(\beta,\alpha),
\]
where 
\[
\phi_i=\lambda_{i,1}P_p+\lambda_{i,2}\mathscr{H}P_p,\quad i=1,2,3,
\]
for a collection of full spark vectors 
\[
\{\lambda_i=(\lambda_{i,1},\lambda_{i,2})\}_{i=1}^{3}.
\]
We prove that $f=\pm g$. 
Let us first observe that, for every  $(b,a)\in\Lambda(\beta,\alpha)$ and for every  $i=1,2,3$, it holds that 
\begin{align*}
    |\langle(\mathcal{W}_{P_p} f(b,a),\mathcal{W}_{\mathscr{H}P_p} f(b,a)),\lambda_i\rangle|&=|\mathcal{W}_{\phi_i} f(b,a)| = |\mathcal{W}_{\phi_i} g(b,a)|\\
    &=|\langle(\mathcal{W}_{P_p} g(b,a),\mathcal{W}_{\mathscr{H}P_p} g(b,a)),\lambda_i\rangle|.
\end{align*}
By hypothesis, the set of vectors $\{\lambda_i\}_{i=1}^{3}$ is full spark and therefore satisfies the complement property. Consequently, for every $(b,a)\in\Lambda(\beta,\alpha)$,
\begin{align}\label{eq:reference}
(\mathcal{W}_{P_p} f(b,a),\mathcal{W}_{\mathscr{H}P_p} f(b,a))&=\pm(\mathcal{W}_{P_p} g(b,a),\mathcal{W}_{\mathscr{H}P_p} g(b,a)).
\end{align}
If we consider the subset
\begin{align*}
    S=\{(b,a)\in\Lambda(\beta,\alpha):(\mathcal{W}_{P_p} f(b,a),\mathcal{W}_{\mathscr{H}P_p} f(b,a))&=(\mathcal{W}_{P_p} g(b,a),\mathcal{W}_{\mathscr{H}P_p} g(b,a))\},
\end{align*}
then the functions $h_1=f+g$ and $h_2=f-g$ satisfy
\begin{align*}
    (\mathcal{W}_{P_p} h_1(b,a),\mathcal{W}_{\mathscr{H}P_p} h_1(b,a))=0, \qquad (b,a)\in \Lambda(\beta,\alpha)\setminus S
\end{align*}
and
\begin{align*}
    (\mathcal{W}_{P_p} h_2(b,a),\mathcal{W}_{\mathscr{H}P_p} h_2(b,a))=0, \qquad (b,a)\in S.
\end{align*}
Therefore, by Proposition~\ref{prop:cauchyandpoisson}, the functions $h_1$ and $h_2$ satisfy
\begin{align*}
    (2\,\text{Re}[\mathcal{W}_{\psi_p}h_1(b,a)],-2\,\text{Im}[\mathcal{W}_{\psi_p}h_1(b,a)])=0, \qquad (b,a)\in \Lambda(\beta,\alpha)\setminus S
\end{align*}
and
\begin{align*}
    (2 \,\text{Re}[\mathcal{W}_{\psi_p}h_2(b,a)],-2 \,\text{Im}[\mathcal{W}_{\psi_p}h_2(b,a)])=0, \qquad (b,a)\in S,
\end{align*}
or equivalently
\begin{align*}
    \mathcal{W}_{\psi_p}h_1(b,a)=0 \qquad (b,a)\in \Lambda(\beta,\alpha)\setminus S \qquad \text{and} \qquad \mathcal{W}_{\psi_p}h_2(b,a)=0 \qquad (b,a)\in S.
\end{align*}
We introduce the functions  
\begin{align*}
H_1(b+ia):=a^{-(\frac{1}{2}+p)}\mathcal{W}_{\psi_p}h_1(b,a),\quad H_2(b+ia):=a^{-(\frac{1}{2}+p)}\mathcal{W}_{\psi_p}h_2(b,a),
\end{align*}
of one complex variable, where $(b,a)\in\R\times\R_+$, and note that $H_1$ and $H_2$ are entire functions on the upper half plane (cf.~equation~\eqref{eq:mapbergmanwavelet}). As a consequence, the function
\[
    H=H_1\cdot H_2
\]
is also an entire function on the upper half plane and 
\begin{equation}\label{eq:Hlattice}
    H(b+ia)=0, \qquad (b,a)\in\Lambda(\beta,\alpha).
\end{equation}
Furthermore, by equation~\eqref{eq:waveletisometry},
\begin{align*}
    \int_{\R\times\R_+}|H(b+ia)|^2a^{2+4p}\frac{{\rm d}b{\rm d}a}{a^2}&=\int_{\R\times\R_+}|\mathcal{W}_{\psi_p}h_1(b,a)|^2|\mathcal{W}_{\psi_p}h_2(b,a)|^2\frac{{\rm d}b{\rm d}a}{a^2}\\
    &\leq \|\psi_p\|_2^2\|h_1\|_2^2\int_{\R\times\R_+}|\mathcal{W}_{\psi_p}h_2(b,a)|^2\frac{{\rm d}b{\rm d}a}{a^2}\\
    &=C_{\psi_p}\|\psi_p\|_2^2\|h_1\|_2^2\|h_2\|_2^2,
\end{align*}
where
\[
C_{\psi_p}=\int_{0}^{+\infty}\frac{|\widehat{\psi_p}(\xi)|^2}{\omega}{\rm d}\omega<\infty,
\]
showing that $H$ belongs to the weighted Bergman space $B_{2+4p}(\mathbb{C}_+)$. By Theorems~\ref{thm:seip1} and \ref{thm:seip}, the image of $\Lambda(\beta,\alpha)$ via the isomorphism $\R^2\ni(x,y)\mapsto x+iy\in\mathbb{C}$ is a uniqueness set for $B_{2+4p}(\mathbb{C}_+)$ and hence equation~\eqref{eq:Hlattice} implies $H\equiv 0$. Therefore, either $H_1 \equiv 0$ or $H_2 \equiv 0$ which shows that $f = \pm g$.
\end{proof}

\begin{ex}
Theorem~\ref{thm:multiwavelets} applies to the analyzing wavelets
\[
\phi_1=P,\quad \phi_2=\mathscr{H}P,\quad \phi_3=P+\mathscr{H}P,
\]
together with the dyadic lattice $\Lambda(4\pi/5\ln(2),2)$, see Figures~\ref{fig:fandg1} and \ref{fig:fandg2}.
\end{ex}
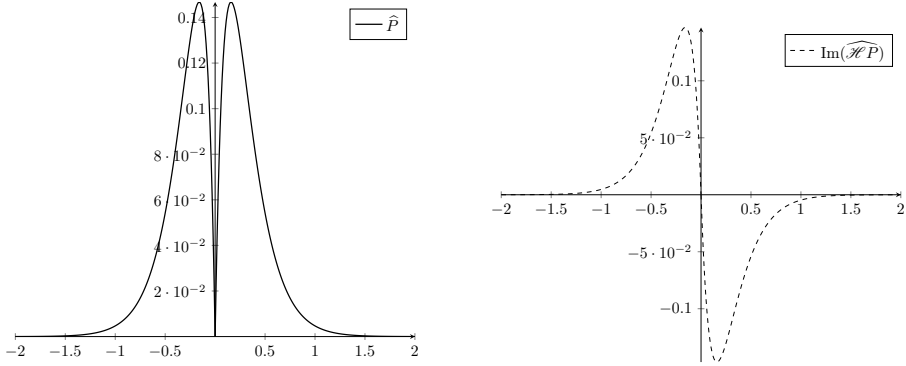
\begin{figure}
    \centering
        \begin{tikzpicture}[scale=0.6]
        \begin{axis}[height=9cm, axis lines = middle]
            \addplot[domain=0:2, samples=200, thick] {sqrt(2*pi)*x*exp(-2*pi*x)};
            \addplot[domain=-2:0, samples=200, thick] {-sqrt(2*pi)*x*exp(2*pi*x)};
        \addlegendentry{$\widehat{P}$}
        \end{axis}
    \end{tikzpicture}
      \qquad
    \begin{tikzpicture}[scale=0.6]
        \begin{axis}[height=9cm, axis lines = middle]
            \addplot[domain=0:2, samples=200, dashed] {-sqrt(2*pi)*x*exp(-2*pi*x)};
            \addplot[domain=-2:0, samples=200, dashed] {-sqrt(2*pi)*x*exp(2*pi*x)};
        \addlegendentry{Im($\widehat{\mathscr{H}P}$)}
       \end{axis}
 \end{tikzpicture}
     \caption{The Fourier transform of the Poisson wavelet (on the left) and its Hilbert transform (on the right). Since the Fourier transform of the Poisson wavelet is real-valued, the Fourier transform of its Hilbert transform is purely imaginary.}
     \label{fig:fandg2}
\end{figure}
\begin{rmk}\label{rmk:nonuniformsampling}
More generally, we could replace the regular set of sampling $\Lambda(\beta,\alpha)$ in Theorem~\ref{thm:multiwavelets} with any other sampling set for $B_{2+4p}(\mathbb{C}_+)$. We refer to \cite{seip199nonregular} for a complete characterization of sets of sampling for Bergman type spaces on the unit disk $\mathbb{D}$. We recall that the upper-half plane $\mathbb{C}_+$ can be mapped onto the unit disk $\mathbb{D}$ under the conformal map
\[
\varphi(z)=\frac{z-i}{z+i},
\]
known as the Cayley transform. 
\end{rmk}
\begin{rmk}\label{rmk:cpcauchy}
In the second part of the proof of Theorem~\ref{thm:multiwavelets}, we actually show that, if $\beta\ln(\alpha)\leq4\pi/(1+4p)$, the Cauchy wavelet frame
\[
\{T_bD_a\psi_p\}_{(b,a)\in\Lambda(\beta,\alpha)}
\]
for the Hardy space $\mathcal{H}_+(\R)$ satisfies the complement property (see Definition~\ref{def:cphilbert}). Indeed, suppose that there exists a subset $S\subseteq \Lambda(\beta,\alpha)$ such that 
\[
\overline{\text{span}\{T_bD_a\psi_p\}_{(b,a)\in \Lambda(\beta,\alpha)\setminus S}}\ne\mathcal{H}_+(\R)\quad\text{and}\quad\overline{\text{span}\{T_bD_a\psi_p\}_{(b,a)\in S}}\ne\mathcal{H}_+(\R).
\]
This means that there exist two non-zero functions $h_1,h_2\in\mathcal{H}_+(\R)$ such that 
\begin{align*}
    \mathcal{W}_{\psi_p}h_1(b,a)=0 \qquad (b,a)\in \Lambda(\beta,\alpha)\setminus S \qquad \text{and}\qquad \mathcal{W}_{\psi_p}h_2(b,a)=0 \qquad (b,a)\in S.
\end{align*} 
Then, the same argument as in the proof of Theorem~\ref{thm:multiwavelets} leads to the conclusion that either $h_1$ or $h_2$ has to be zero and thus the complement property holds. However, in the complex-valued case, the complement property is only a necessary condition for uniqueness of the phase retrieval problem and we cannot conclude that the operator 
\begin{equation}\label{eq:properatorcauchy}
\mathcal{H}_+(\R)/\!\sim\, \ni h\mapsto (|\mathcal{W}_{\psi_p} h(b,a)|)_{(b,a)\in\Lambda(\beta,\alpha)} \in \bbR^{\Lambda(\beta,\alpha)}
\end{equation}
is injective. Determining a hyperbolic lattice such that the operator in equation~\eqref{eq:properatorcauchy} is injective remains an interesting open problem.
\end{rmk}
\section*{Acknowledgment}
The authors would like to thank Marco Peloso for pointing them to the reference \cite{seip199nonregular} discussed in Remark~\ref{rmk:nonuniformsampling}.

\bibliographystyle{amsplain}

\end{document}